\documentclass[12pt]{amsart}
\usepackage{epsfig,color}

\makeatother

\theoremstyle{definition}

\newtheorem{thm}{Theorem}[section]
\newtheorem{cor}[thm]{Corollary}

\newtheorem{lem}[thm]{Lemma}
\newtheorem{con}[thm]{Conjecture}
\newtheorem{exa}[thm]{Example}
\newtheorem{rmk}[thm]{Remark}
\newtheorem{prop}[thm]{Proposition}

\numberwithin{equation}{section}

\address{MIT, Dept. of Math.\\
77 Massachusetts Avenue, Cambridge, MA 02139-4307.}
\email{jchang61@mit.edu}

\title{1-dimensional solutions of the $\lambda$-self shrinkers}
\author{Jui-En Chang} 
\date{\today}

\usepackage{latexsym, amsmath,amssymb}
\usepackage{amsmath}
\usepackage{amsfonts}
\usepackage{amssymb}
\usepackage{amscd}
\usepackage{amsthm}
\usepackage{amsbsy}
\usepackage{graphicx}
\usepackage{bm}
\usepackage{epsfig,epstopdf,url,array}

\begin{document}
\maketitle
\begin{abstract} 
We examine the solutions of 1-dimensional $\lambda$-self shrinkers and show that for certain $\lambda<0$, there are some closed, embedded solutions other than circles. For negative $\lambda$ near zero, there are embedded solutions with 2-symmetry. For negative $\lambda$ with large absolute value, there are embedded solutions with $m$-symmetry, where $m$ is greater than 2.
\end{abstract}

\section{Introduction}
We consider the $\lambda$-hypersurface equation
\begin{equation}\label{lambda}
H=\frac{\langle x,N\rangle}{2}+\lambda,
\end{equation}
where $H$ is the mean curvature, $N$ is the normal vector on the surface and $\lambda$ is a constant.

This equation is first studied by McGonagle and Ross\cite{MR} and be named as $\lambda$-hypersurface in the work of Cheng and Wei\cite{CW1}. The equation arises in the Gaussian isoperimetric problem: In $\mathbb{R}^{n+1}$, the weighted Gaussian volume element and area element are given by $dV_\mu=\exp(-\frac{|x|^2}{4})dV$ and  $dA_\mu=\exp(-\frac{|x|^2}{4})dA$, where $dV$ and $dA$ are the volume element and area element induced by the Euclidean metric. The Gaussian isoperimetric problem asks: Among all regions with the same weighted volume $V_0$, which one has the least weighted boundary area? The answer to this problem is given in \cite{B}, \cite{ST}:  The half space minimizes the weighted boundary area.

The problem above can be considered locally as follows. Let $\Sigma$ be a surface minimizing the weighted boundary area among all surfaces enclosing the same weighted volume. The surface is a critical point for all variations that fix the enclosed weighted volume. The geometry condition on $\Sigma$ is given by equation \eqref{lambda}. This equation is defined on $\Sigma$ locally and it can be studied even if $\Sigma$ does not enclose a region. The solutions can be thought of as the critical points to the weighted area functional. The reader can refer to \cite{MR} for more detail.

The $\lambda$-hypersurfaces also arise in the study of the weighted volume-preserving flow by Cheng and Wei\cite{CW1}. Note that the ``weighted volume" in their work is defined on the surface and is different from the Gaussian weighted volume above. They prove that $\lambda$-hypersurfaces are critical points of the weighted area functional for all weighted volume-preserving variations. The $\lambda$-hypersurfaces can also be characterized as having constant weighted mean curvature.

\begin{exa}\label{examples}
The following are $\lambda$-hypersurfaces in $\mathbb{R}^{n}$:
\begin{enumerate}
\item The hyperplane $\mathbb{R}^{n-1}$ which is $|\lambda|$ away from the origin.
\item The cylinder $\mathbb{S}^k(r)\times\mathbb{R}^{n-1-k}$ with $\mathbb{S}^k(r)$ centered at the origin and

$r=\sqrt{\lambda^2+2k}-\lambda$.
\item The sphere $\mathbb{S}^{n-1}(r)$ centered at the origin with $r=\sqrt{\lambda^2+2(n-1)}-\lambda$.
\end{enumerate}
\end{exa}

The examples above admit properties such as polynomial volume growth and constant mean curvature. It is important to investigate under which assumption we can deduce a $\lambda$-hypersurface is one of the above. Some rigidity results can be found in \cite{COW}, \cite{CW2}, \cite{G} and \cite{MR}.

Note that in the special case $\lambda=0$, equation \eqref{lambda} becomes
\begin{equation}\label{shrinker}
H=\frac{\langle x,N\rangle}{2},
\end{equation}
which is the self-shrinker equation in mean curvature flow. This comes from the fact that the self-shrinkers are also the critical points of the weighted area functional in Gaussian space. Therefore, we call equation \eqref{lambda} the $\lambda$-self shrinker equation in this paper. In the study of mean curvature flow, solutions of equation \eqref{shrinker} are the solitons which move by scaling with respect to the origin. They are models for singularities in the study of mean curvature flow.

For 1-dimensional self-shrinkers in $\mathbb{R}^2$, Abresch and Langer\cite{AL} completely characterized the closed solutions. They prove that circles are the only closed, embedded solutions. For higher dimensional case, the most well-known solution is the ``Angenent's doughnuts" in $\mathbb{R}^3$, see \cite{A}. M{\o}ller\cite{M1} constructs more closed embedded solutions. For noncompact examples, the reader can refer to \cite{KM}. Some of the most improtant literature concerning the classification of the self-shrinkers is the following: Huisken\cite{H1} classifies the self-shrinkers under the condition of bounded second fundamental form, nonnegative mean curvature and polynomial volume growth. Later, Colding and Minicozzi\cite{CM} remove the requirement of the second fundamental form bound.

This paper focuses on the behavior of the equation \eqref{lambda} in $\mathbb{R}^2$. To simplify the equation, we scale the curve by a factor of $\sqrt{2}$ to make the constant $\frac{1}{2}$ become 1 and use the 1-dimensional curvature $k$ in place of mean curvature $H$. The equation becomes
\begin{equation}\label{shrinkereq}
k=-\langle x,N\rangle+\lambda.
\end{equation}
We give a sufficient constraint on $\lambda$ for embedded solutions to exist. Recall that if a set in $\mathbb{R}^2$ is invariant under the rotation of the angle $\frac{2\pi}{m}$ with respect to the origin, we say the set admits $m$-symmetry.
\begin{thm}\label{neglambdanear0}
For $\frac{-2}{\sqrt{3}}<\lambda<0$, there exists an embedded solution with 2-symmetry.
\end{thm}
\begin{thm}\label{neglambdanearinfty}
There exist $\delta>0$ such that for $\lambda<\frac{-7}{2\sqrt{2}}+\delta$, there exists an embedded solution with $m$-symmetry, $m>2$.
\end{thm}

Unlike the result of Abresch and Langer\cite{AL} that for the $\lambda=0$ case, the circle is the only closed embedded solution, we surprisingly find other embedded solutions. This affects the understanding of the rigidity problem about the classification of $\lambda$-hypersurfaces. If we cross the curve with $\mathbb{R}^{n-1}$, we obtain a $\lambda$-hypersurface in $\mathbb{R}^{n+1}$ which is topologically $\mathbb{S}^1\times\mathbb{R}^{n-1}$ with non-vanishing mean curvature and polynomial area growth. However, this is not the standard round cylinder as in example \ref{examples}. It is also interesting to compare the result with the isoperimetric problem in Euclidean space, where the critical surface to the area functional should have constant mean curvature. Thus the only 1-dimensional solutions of the isoperimetric problem in $\mathbb{R}^2$ are circles.

For positive $\lambda$, the behavior is similar to the self-shrinking curves. We have
\begin{thm}\label{poslambda}
When $\lambda>0$, there is no embedded solutions to the equation \eqref{shrinkereq} with $k>0$.
\end{thm}
\begin{rmk}It is worth mentioning that Guang\cite{G} establishes the same result as in theorem \ref{poslambda} with different proof. He considers the part of the curve where the curvature decrease from the maximum to the minimum. 
\end{rmk}


This paper will be structured as follows: In section 2, starting from the defining equation, we derive an ODE system for the 1-dimensional $\lambda$-self shrinkers. The approach used here is similar to that in \cite{H2}. In section 3, we analyze the behavior of the solution for the extreme cases: The energy is near the minimum and and energy is near infinity. We prove the theorem \ref{neglambdanear0}, \ref{neglambdanearinfty} and \ref{poslambda}. In section 4, we use Matlab to get some curves which approximately solve the equation. Some pictures of the curves are provided for a better understanding of the behavior for each of the different cases in the main theorems.

\section{Setting up the ODE system}
For a curve $x(s)\in\mathbb{R}^2$ parametrized by arc length $s$, we have
\begin{equation}
\begin{split}
\frac{d}{ds}x&=T,\\
\frac{d}{ds}T&=kN,
\end{split}
\end{equation}
where $T$ and $N$ are the tangent vector and the normal vector of the curve, respectively. Note that for any curve in $\mathbb{R}^2$, we have two possible choices of $N$: either rotate $T$ clockwise by $\frac{\pi}{2}$ or $-\frac{\pi}{2}$. If we let $N^-=-N$, $k^-=-k$, we have $kN=k^-N^-$. Therefore, we have
\begin{equation}
k^-=-k=\langle x,N\rangle-\lambda=-\langle x,N^-\rangle-\lambda.
\end{equation}
This tells us that selecting the opposite normal vector will change the sign of $k$ and result in a solution corresponding to $-\lambda$.

Using the method in \cite{H2}, put $\tau=\langle x,T\rangle$, $\nu=\langle x,N\rangle$. We can obtain the ODE system
\begin{equation}
\begin{split}
\frac{d}{ds}\tau&=1+k\nu=1-\nu^2+\lambda\nu,\\
\frac{d}{ds}\nu&=-k\tau=\nu\tau-\lambda\tau.
\end{split}
\end{equation}
The equilibrium is the point where $\frac{d}{ds}\tau=\frac{d}{ds}\nu=0$. They are given by $(0,\nu^\pm)$, where $\nu^\pm$ are the positive and the negative solutions of the equation $\nu^2-\lambda\nu-1=0$, respectively. At the equilibrium, the curvature is a nonzero constant. It corresponds to the circle centered at the origin. For $(0,\nu^+)$, it is a circle of radius $\nu^+$ with the normal pointed outward and $k<0$. For $(0,\nu^-)$, it is a circle of radius $-\nu^-=|\nu^-|$ with the normal pointed inward and $k>0$. Also, note that $(\tau,\nu)=(s,\lambda)$ is a solution which corresponds to a line with the minimum distance to the origin equal to $\lambda$. From now on, without loss of generality, we only consider solutions with $k\geq0$. They are the solutions with the trajectory contained in the half plane $\{\nu\leq\lambda\}$ of $\tau-\nu$ space. For the solutions with $k<0$, choose the opposite normal vector and study them as the solutions corresponding to $-\lambda$ with positive $k$.

\subsection{Periodicity of the solution}
For a solution to the system, the function 
\begin{equation}
F(\tau,\nu)=(\lambda-\nu)\exp(-\frac{\nu^2+\tau^2}{2})
\end{equation}
is positive in the $\{\nu\leq\lambda\}$ half plane. Differentiating it with respect to $s$, we have
\begin{equation}
\begin{split}
\frac{d}{ds}F&=\big(-\frac{d}{ds}\nu-(\lambda-\nu)(\nu\frac{d}{ds}\nu+\tau\frac{d}{ds}\tau)\big)\exp(-\frac{\nu^2+\tau^2}{2})\\
&=\Big(-(\nu-\lambda)\tau+(\nu-\lambda)\big(\nu(\nu-\lambda)\tau+\tau(1-\nu(\nu
-\lambda))\big)\Big)\exp(-\frac{\nu^2+\tau^2}{2})\\
&=0.
\end{split}
\end{equation}
The trajectory of each solution lies in a level set of $F$. Since each level set of $F$ is a simple closed curve except $\{F=0\}$, which corresponds to the line mentioned before, we have a uniform lower bound away from 0 of the speed of $(\tau(s),\nu(s))$ curve on each level set. Therefore, the solution $(\tau(s),\nu(s))$ should be periodic in $s$.
\begin{rmk}
Note that if $x(s)$ is periodic, then $\tau$, $\nu$ are periodic. But the converse is not true. Starting from a periodic solution of $(\tau(s),\nu(s))$, even though we can reconstruct $x(s)$ by this and the initial condition, the resulting $x(s)$ is periodic only when the change of angle in a period can be expressed as $\frac{n}{m}2\pi$, where $n$, $m$ are relatively prime positive integers. In this case, the period of $x(s)$ is $m$ times the period of $(\tau(s),\nu(s))$, and it will be a closed solution.
\end{rmk}

\subsection{Change of angle in a period}
Now, since $k$ is more directly related to the geometric behavior than $\nu$, we use $(\tau,k)$ as the variable instead of $(\tau,\nu)$. From now on, unless otherwise specified, we use $'$ for $\frac{d}{ds}$. Plugging $\nu=\lambda-k$ into the previous ODE system, it becomes
\begin{equation}
\begin{split}
\tau'&=1+\lambda k-k^2,\\
k'&=k\tau.
\end{split}
\end{equation}
and $\{\nu\leq\lambda\}$ becomes $\{k\geq0\}$ in $\tau-k$ plane. Note that after the change of variable, we still have the equilibrium at $(0,k^\pm)$, where $k^\pm=\nu^\pm$ because they satisfy exactly the same equation. However, $(0,k^\pm)$ correspond to $(0,\nu^\mp)$, respectively. The $\nu=\lambda$ line in $\tau-\nu$ space now becomes $k=0$ line in $\tau-k$ space.

In the $\{k>0\}$ half space, let $B=2\log k$. We have $B'=2\tau$ and 
\begin{equation}
B''=2\tau'=2+2k(\lambda-k)=2+2\lambda e^\frac{B}{2}-2e^B.
\end{equation}
Multiplying both sides by $B'$ and integrating with respect to $s$, we get
\begin{equation}
\frac{1}{2}(B')^2+2e^B-4\lambda e^\frac{B}{2}-2B=-4\log F-2\lambda^2.
\end{equation}
If we define $F_\lambda=F\cdot \exp\frac{\lambda^2}{2}$, $V(B)=e^B-2\lambda e^\frac{B}{2}-B$, the equation becomes
\begin{equation}
\frac{1}{2}(B')^2+2V(B)=-4\log F_\lambda.
\end{equation}
The minimum of $V(B)$ is attained when $\frac{d}{dB}V(B)=0$. $e^B-\lambda e^\frac{B}{2}-1=0$. $e^\frac{B}{2}=k^+$. This corresponds to the equilibrium at $(0,k^+)$ and $\min V(B)=-\lambda k^+-2\log k^++1$. Now, for any $\eta>\min V(B)$, we can find the solutions $B_\eta^-<B_\eta^+$ of $V(B)=\eta$. Considering the differential equation of $B$, we get
\begin{equation}
\begin{split}
\frac{1}{2}(B')^2&+2V(B)=2\eta,\\
B'&=\pm 2\sqrt{\eta-V(B)}.
\end{split}
\end{equation}
Therefore, the length of the curve in a period is given by
\begin{equation}
\oint ds=2\int_{B_\eta^-}^{B_\eta^+}(\frac{dB}{ds})^{-1}dB=\int_{B_\eta^-}^{B_\eta^+}\frac{1}{\sqrt{\eta-V(B)}}dB,
\end{equation}
and the change of the angle in a period is given by 
\begin{equation}
\triangle\theta=\oint kds=\int_{B_\eta^-}^{B_\eta^+}\frac{e^\frac{B}{2}}{\sqrt{\eta-V(B)}}dB.
\end{equation}
In order to simplify the calculation, let $u=e^\frac{B}{2}$, $u_\eta^\pm=e^\frac{B_\eta^\pm}{2}$, respectively. $V(u)=u^2-2\lambda u-2\log u$ after the change of variable. The change of angle, $\triangle\theta$, is given by
\begin{equation}
\triangle\theta=\int_{u_\eta^-}^{u_\eta^+}\frac{2du}{\sqrt{\eta-V(u)}}.
\end{equation}

\section{The behavior of the solutions}

Now, we will focus on the behavior of $\triangle\theta$ when the energy $\eta$ varies from $\min V(B)$ to $\infty$.

\subsection{The behavior of the solution when $\eta$ is near min$V(B)$}
The following focuses on the behavior of $\triangle\theta$. When the energy is near the minimum, the behavior is closed to a simple harmonic oscillator. 
\begin{lem}
For any potential function $V\in C^2$, at a local minimum $x_0$ with positive second derivative, let $u_\eta^-$ be the largest solution of $V(u)=\eta$ which is below $x_0$ and let $u_\eta^+$ be the smallest solution of $V(u)=\eta$ which is above $x_0$. We have
\begin{equation}
\lim_{\eta\to V(x_0)^+}\int_{u_\eta^-}^{u_\eta^+}\frac{du}{\sqrt{\eta-V(u)}}=\sqrt{\frac{2}{V''(x_0)}}\pi.
\end{equation}
\end{lem}
\begin{proof}
First, note that for the case in which the potential is quadratic, $V(u)=V(x_0)+\frac{V''(x_0)(x-x_0)^2}{2}$, a simple calculation shows that
\begin{equation}
\int_{u_\eta^-}^{u_\eta^+}\frac{du}{\sqrt{\eta-V(u)}}=\sqrt{\frac{2}{V''(x_0)}}\pi
\end{equation}
for any $\eta>V(x_0)$ and is independent of $\eta$.

For arbitrary potential function $V\in C^2$ and $\epsilon>0$, there is $\delta>0$ such that for all $V(x_0)<\eta<V(x_0)+\delta$, we have $|V''(u)-V''(x_0)|<\epsilon$ for $u\in[u_\eta^-,u_\eta^+]$. Let $V_\pm$ be the quadratic function which pass through $(u_\eta^-,\eta)$, $(u_\eta^+,\eta)$ with $V_\pm''=V''(x_0)\mp\epsilon$. We have $V_-<V<V_+$ in $(u_\eta^-, u_\eta^+)$. Therefore,
\begin{equation}
\begin{split}
\sqrt{\frac{2}{V''(x_0)+\epsilon}}\pi&=\int_{u_\eta^-}^{u_\eta^+}\frac{du}{\sqrt{\eta-V_-(u)}}\leq\int_{u_\eta^-}^{u_\eta^+}\frac{du}{\sqrt{\eta-V(u)}}\\
&\leq\int_{u_\eta^-}^{u_\eta^+}\frac{du}{\sqrt{\eta-V_+(u)}}=\sqrt{\frac{2}{V''(x_0)-\epsilon}}\pi.
\end{split}
\end{equation}
Letting $\epsilon$ goes to 0 yields the desired result.
\end{proof}

\begin{prop}
When $\eta\to \min V(B)^+$, $\triangle\theta$ approaches $\pi\sqrt{2}\sqrt{\frac{\lambda}{\sqrt{\lambda^2+4}}+1}$. Moreover, $\triangle\theta$ is decreasing in a neighborhood of $\min V(B)$.
\end{prop}
\begin{proof}
Let $\eta\to \min V(B)^+$. The derivatives of $V(u)$ with respect to $u$ at the minimum point are 
\begin{equation}
\begin{split}
V^{(2)}(k^+)&=2+2(k^+)^{-2},\\
V^{(3)}(k^+)&=-4(k^+)^{-3},\\
V^{(4)}(k^+)&=12(k^+)^{-4}.\\
\end{split}
\end{equation}
Therefore, from the lemma above and recall that $k^+=\frac{\lambda+\sqrt{\lambda^2+4}}{2}$, we have
\begin{equation}
\begin{split}
\lim_{\eta\to\min V(B)^+}\triangle\theta&=\lim_{\eta\to\min V(B)^+}\int_{u_\eta^-}^{u_\eta^+}\frac{2du}{\sqrt{\eta-V(u)}}\\
&=2\pi\cdot\sqrt{\frac{2(k^+)^2}{2(k^+)^2+2}}=\pi\sqrt{2}\sqrt{\frac{\lambda}{\sqrt{\lambda^2+4}}+1}.
\end{split}
\end{equation}

From the result of \cite{C}, since
\begin{equation}
\begin{split}
5(V^{(3)})^2-4V^{(2)}V^{(4)}&=80(k^+)^{-6}
-96((k^+)^{-6}+(k^+)^{-4})\\
&=-16(k^+)^{-6}-96(k^+)^{-4}<0,\\
\end{split}
\end{equation}
the function $\triangle\theta$ is monotone decreasing near $\min V(B)$ with respect to $\eta$.
\end{proof}

\begin{rmk}
For the case of self shrinkers, we have $\lambda=0$. The proposition above gives $\triangle\theta\to\sqrt{2}\pi$, as the result in \cite{AL}. This function is strictly increasing with respect to $\lambda$. When $\lambda$ approaches $\infty$, $\triangle\theta$ approaches $2\pi$. When $\lambda$ approaches $-\infty$, $\triangle\theta$ approaches $0$.
\end{rmk}

\subsection{The behavior of the solution when $\eta$ is near infinity}

Now, we turn our attention to the behavior of $\triangle\theta$ when the energy approaches infinity. An upperbound of $\triangle\theta$ is given by the following proposition.

\begin{prop}
For any $L>1$,
\begin{equation}
\triangle\theta\leq\pi+2(\lambda-1+\sqrt{\frac{L}{L-1}})\frac{1}{\sqrt{\eta}}+o(\frac{1}{\sqrt{\eta}})
\end{equation}
as $\eta$ goes to infinity.
\end{prop}
\begin{proof}
In order to get an upper bound of $\triangle\theta$, separate the integration into two terms.
\begin{equation}
\triangle\theta=\int_{u_\eta^-}^{u_\eta^+}\frac{2du}{\sqrt{\eta-V(u)}}=\int_{u_\eta^-}^1\frac{2du}{\sqrt{\eta-V(u)}}+\int_1^{u_\eta^+}\frac{2du}{\sqrt{\eta-V(u)}}.
\end{equation}
When $1\leq u\leq u_\eta^+$, let $\bar{u}_\eta^+$ be the positive solution of $\eta=u^2-2\lambda u$. Note that $\bar{u}_\eta^+<u_\eta^+$. Let $\bar{V}(u)=(\frac{\bar{u}_\eta^+-1}{u_\eta^+-1}(u-1)+1)^2-2\lambda(\frac{\bar{u}_\eta^+-1}{u_\eta^+-1}(u-1)+1)$. At $u=u_\eta^+$, $V(u)=\bar{V}(u)=\eta$. At $u=1$, $V(u)=\bar{V}(u)=1-2\lambda$. The second derivative of $\bar{V}(u)-V(u)$ is
\begin{equation}
\Big(\bar{V}(u)-V(u)\Big)''=2(\frac{\bar{u}_\eta^+-1}{u_\eta^+-1})^2-(2+2\frac{1}{u^2})<0.
\end{equation}
We can conclude
\begin{equation}
\bar{V}(u)-V(u)\geq 0
\end{equation}
for any $1<u<u_\eta^+$. Therefore, we have
\begin{equation}
\begin{split}
\int_1^{u_\eta^+}\frac{2du}{\sqrt{\eta-V(u)}}&\leq\int_1^{u_\eta^+}\frac{2du}{\sqrt{\eta-\bar{V}(u)}}\\
&=\frac{u_\eta^+-1}{\bar{u}_\eta^+-1}\int_1^{\bar{u}_\eta^+}\frac{2dv}{\sqrt{\eta-v^2+2\lambda v}}\\
&=2\frac{u_\eta^+-1}{\bar{u}_\eta^+-1}(\frac{\pi}{2}-\sin^{-1}\frac{1-\lambda}{\sqrt{\eta+\lambda^2}}).\\
\end{split}
\end{equation}

We need an upper bound for $\frac{u_\eta^+-1}{\bar{u}_\eta^+-1}$. Starting from $\bar{u}_\eta^+=\lambda+\sqrt{\lambda^2+\eta}$, $V(\bar{u}_\eta^+)=\eta-2\log\bar{u}_\eta^+$, $V(u_\eta^+)=\eta$ and $V'(u)\geq 2\bar{u}_\eta^+-2\lambda-2\frac{1}{\bar{u}_\eta^+}$ for $\bar{u}_\eta^+<u<u_\eta^+$, we have
\begin{equation}
u_\eta^+-\bar{u}_\eta^+\leq\frac{2\log\bar{u}_\eta^+}{2\bar{u}_\eta^+-2\lambda-2\frac{1}{\bar{u}_\eta^+}}=\frac{\log\bar{u}_\eta^+}{\bar{u}_\eta^+-\lambda-\frac{1}{\bar{u}_\eta^+}}\leq\frac{C\log\eta}{\frac{\sqrt{\eta}}{2}}=O(\eta^{-\frac{1}{2}}\log\eta)
\end{equation}
for $\eta$ large enough. Hence,
\begin{equation}
\frac{u_\eta^+-1}{\bar{u}_\eta^+-1}=1+\frac{u_\eta^+-\bar{u}_\eta^+}{\bar{u}_\eta^+-1}=1+O(\eta^{-1}\log\eta).
\end{equation}
Therfore,
\begin{equation}
\int_1^{u_\eta^+}\frac{2du}{\sqrt{\eta-V(u)}}\leq2\frac{u_\eta^+-1}{\bar{u}_\eta^+-1}\big(\frac{\pi}{2}-\sin^{-1}\frac{1-\lambda}{\sqrt{\eta+\lambda^2}}\big)=\pi+2\frac{\lambda-1}{\sqrt{\eta}}+o(\frac{1}{\sqrt{\eta}}).\\
\end{equation}

Now, we are going to estimate the other term. For all $L>1$, let $u_{\eta,L}^-=\exp(-\frac{\eta}{2L}+\frac{1}{2}+|\lambda|)$.  Note that when $\eta$ is large enough, $u_{\eta,L}^-<1$ and $V(u_{\eta,L}^-)<\frac{\eta}{L}$.
\begin{equation}
\int_{u_\eta^-}^1\frac{2du}{\sqrt{\eta-V(u)}}=\int_{u_\eta^-}^{u_{\eta,L}^-}\frac{2du}{\sqrt{\eta-V(u)}}+\int_{u_{\eta,L}^-}^1\frac{2du}{\sqrt{\eta-V(u)}}.
\end{equation}
For the first term, since $V(u_\eta^-)=\eta$, $V(u_{\eta,L}^-)<\frac{\eta}{L}$, $V''>0$, we have 
\begin{equation}
V(u)<\frac{\eta}{L}+(\frac{L-1}{L})\eta\frac{u-u_{\eta,L}^-}{u_\eta^--u_{\eta,L}^-}
\end{equation}
for $u_\eta^-<u<u_{\eta,L}^-$ and
\begin{equation}
\begin{split}
\int_{u_\eta^-}^{u_{\eta,L}^-}\frac{2du}{\sqrt{\eta-V(u)}}&\leq\int_{u_\eta^-}^{u_{\eta,L}^-}\frac{2du}{\sqrt{\frac{L-1}{L}\eta(\frac{u_\eta^--u}{u_\eta^--u_{\eta,L}^-})}}=(u_{\eta,L}^--u_\eta^-)\int_0^1
\frac{2dv}{\sqrt{\frac{L-1}{L}\eta v}}\\
&\leq u_{\eta,L}^-\int_0^1\frac{2dv}{\sqrt{\frac{L-1}{L}\eta v}}= u_{\eta,L}^-\sqrt{\frac{L}{L-1}}\frac{4}{\sqrt{\eta}}.
\end{split}
\end{equation}
The second term can be bounded by the following,
\begin{equation}
\int_{u_{\eta,L}^-}^1\frac{2du}{\sqrt{\eta-V(u)}}\leq \int_{u_{\eta,L}^-}^1\frac{2du}{\sqrt{\frac{L-1}{L}\eta}}\leq\sqrt{\frac{L}{L-1}}\frac{2}{\sqrt{\eta}}.
\end{equation}
Therefore, we get
\begin{equation}
\int_{u_\eta^-}^1\frac{2du}{\sqrt{\eta-V(u)}}\leq\sqrt{\frac{L}{L-1}}\frac{2}{\sqrt{\eta}}+o(\frac{1}{\sqrt{\eta}}).
\end{equation}
Combine the estimation of both terms, we have 
\begin{equation}
\triangle\theta\leq\pi+2\Big(\lambda-1+\sqrt{\frac{L}{L-1}}\Big)\frac{1}{\sqrt{\eta}}+o(\frac{1}{\sqrt{\eta}})
\end{equation}
\end{proof}

After establishing an upper bound of $\triangle\theta$, a lower bound is given by the following:

\begin{prop}
We have
\begin{equation}
\triangle\theta\geq\pi+\sin^{-1}\frac{\lambda-u_\eta^-}{\sqrt{\eta+2\log u_\eta^++\lambda^2}}.
\end{equation}
\end{prop}
\begin{proof}
To get a lower bound of $\triangle\theta$, use $\log u\leq\log u_\eta^+$ when $u_\eta^-\leq u\leq u_\eta^+$.
\begin{equation}
\begin{split}
\triangle\theta&=\int_{u_\eta^-}^{u_\eta^+}\frac{2du}{\sqrt{\eta-V(u)}}\\
&\geq\int_{u_\eta^-}^{u_\eta^+}\frac{2du}{\sqrt{\eta-u^2+2\lambda u+2\log u_\eta^+}}\\
&=\int_{u_\eta^-}^{u_\eta^+}\frac{2du}{\sqrt{(\eta+2\log u_\eta^++\lambda^2)-(u-\lambda)^2}}\\
&=\pi+\sin^{-1}\frac{\lambda-u_\eta^-}{\sqrt{\eta+2\log u_\eta^++\lambda^2}}.
\end{split}
\end{equation}
\end{proof}

Combining both the upper bound and the lower bound, we can get the limit of $\triangle\theta$ when the energy $\eta$ goes to infinity.

\begin{prop}
When the energy $\eta$ goes to infinity, \begin{equation}
\lim_{\eta\to\infty}\triangle\theta=\pi.
\end{equation}
\end{prop}
\begin{proof}
As $\eta$ goes to infinity, $u_\eta^+$ goes to infinity and $u_\eta^-$ goes to zero. Therefore, the upper bound and the lower bound established above both approach $\pi$ as $\eta$ goes to infinity.
\end{proof}

\begin{cor}
When $\lambda<0$, $\triangle\theta<\pi$ for $\eta$ large enough.
\end{cor}
\begin{proof}
Since
\begin{equation}
\triangle\theta\leq\pi+2\Big(\lambda-1+\sqrt{\frac{L}{L-1}}\Big)\frac{1}{\sqrt{\eta}}+o(\frac{1}{\sqrt{\eta}})
\end{equation}
for arbitrary $L>1$, choose $L$ large enough so that $\lambda-1+\sqrt{\frac{L}{L-1}}<0$.
\end{proof}

With the knowledge of the behavior of $\triangle\theta$ for small energy and large energe, we can proof the theorem concerning the case $\lambda<0$.

\begin{proof}[Proof of theorem \ref{neglambdanear0}]
For $\frac{-2}{\sqrt{3}}<\lambda$, when $\eta\to\min V(B)^+$, the limit of $\triangle\theta$ is greater than $\pi$. For $\lambda<0$, $\triangle\theta<\pi$ when $\eta$ is large enough. From continuity, there exists $\eta$ such that $\triangle\theta$ is exactly $\pi$.
\end{proof}

\begin{proof}[Proof of Theorem \ref{neglambdanearinfty}]
For any $\lambda<\frac{-7}{2\sqrt{2}}$, when $\eta\to\min V(B)^+$, the limit of $\triangle\theta$ is less than $\frac{2\pi}{3}$. When $\eta$ is large enough, $\triangle\theta$ approaches $\pi$. From continuity, there exist $\eta$ such that $\triangle\theta$ is exactly $\frac{2\pi}{m}$ for some integer $m>2$.

$\triangle\theta$ is decreasing when $\eta$ is near $\min V(B)$. 
When $\lambda=\frac{-7}{2\sqrt{2}}$, $\min_\eta\triangle\theta<\frac{2\pi}{3}$. From the continuity, the range for $\lambda$ can be extend a little higher than $\frac{-7}{2\sqrt{2}}$.
\end{proof}

\subsection{Relation between $\lambda$ and $\triangle\theta$}
For the case $\lambda>0$, the behavior is similar to the original case for self shrinking curve in Abresch and Langer's paper. We want to compare the change of angle with the case $\lambda=0$.

For simplicity, use $k$ for $k^+$ which is a function of $\lambda$. Translate the minimum point of $V_\lambda(u)$ to the origin, define $\hat{V}_\lambda(u)=V_\lambda(u+k)-\min V_\lambda$, where $\min V_\lambda=V_\lambda(k)=k^2-2\lambda k-2\log k$. Let $\bar{\eta}=\eta-\min V_\lambda$ be the energy relative to the minimum. We have the following theorem:

\begin{thm}
With the setting above, for a fixed $\bar{\eta}$, $\triangle\theta$ is increasing with respect to $\lambda$.
\end{thm}
\begin{proof}
Note that in this setting,
\begin{equation}
\triangle\theta=\int^{u^+_{\bar{\eta},\lambda}}_{u^-_{\bar{\eta},\lambda}}\frac{2du}{\sqrt{\bar{\eta}-\hat{V}_\lambda(u)}},
\end{equation}
where $u^\pm_{\bar{\eta},\lambda}$ are the positive and negative solution of $\bar{\eta}=\hat{V}_\lambda(u)$, respectively.

Now, for a fixed $\bar{\eta}$, we want to know the relation between $\lambda$ and $u$ when $\bar{\eta}=\hat{V}_\lambda(u)$. Differentiate the equation with respect to $\lambda$, we have
\begin{equation}
\begin{split}
0&=2\Big((u+k)-\lambda-\frac{1}{u+k}\Big)(\frac{du}{d\lambda}+\frac{dk}{d\lambda})-2(u+k)-2(k-\lambda-\frac{1}{k})\frac{dk}{d\lambda}+2k\\
&=2\frac{(u+k)^2-\lambda(u+k)-1}{u+k}(\frac{du}{d\lambda}+\frac{dk}{d\lambda})-2u\\
&=2\frac{u^2+2ku-\lambda u}{u+k}(\frac{du}{d\lambda}+\frac{dk}{d\lambda})-2u\\
&=2u\Big[\frac{u+2k-\lambda}{u+k}(\frac{du}{d\lambda}+\frac{dk}{d\lambda})-1\Big].
\end{split}
\end{equation}
Therefore,
\begin{equation}
\frac{du}{d\lambda}+\frac{dk}{d\lambda}=\frac{u+k}{u+2k-\lambda}.
\end{equation}
Since $k=\frac{\lambda+\sqrt{\lambda^2+4}}{2}$, we have
\begin{equation}
\frac{dk}{d\lambda}=\frac{1+\frac{\lambda}{\lambda^2+4}}{2}=\frac{\sqrt{\lambda^2+4}+\lambda}{2\sqrt{\lambda^2}+4}=\frac{k}{2k-\lambda},
\end{equation}
and
\begin{equation}
\frac{du}{d\lambda}=\frac{u+k}{u+2k-\lambda}-\frac{k}{2k-\lambda}=\frac{(k-\lambda)u}{(u+2k-\lambda)(2k-\lambda)}.
\end{equation}
Since $k-\lambda$, $2k-\lambda$, $u+k$ are all positive, $\frac{du}{d\lambda}$ has the same sign as $u$, i.e. $\frac{\partial u^+_{\bar{\eta},\lambda}}{\partial\lambda}>0$, $\frac{\partial u^-_{\bar{\eta},\lambda}}{\partial\lambda}<0$.

Starting from $\hat{V}'_\lambda(u)=2(u+k-\lambda-\frac{1}{u+k})$, for any $\tilde{\eta}<\bar{\eta}$, we want to know the change of the slope of $\hat{V}_\lambda'$ at $u^\pm_{\tilde{\eta},\lambda}$ with respect to $\lambda$. Differentiating the equation with respect to $\lambda$, we have
\begin{equation}
\begin{split}
\frac{d}{d\lambda}(\hat{V}'_\lambda(u))&=2\Big((1+\frac{1}{(u+k)^2})(\frac{du}{d\lambda}+\frac{dk}{d\lambda})-1\Big)\\
&=2\Big((1+\frac{1}{(u+k)^2})\frac{u+k}{u+2k-\lambda}-1\Big)\\
&=\frac{(\lambda-k)u}{(u+k)(u+2k-\lambda)}.
\end{split}
\end{equation}
Note that $\lambda-k<0$ and therefore $\frac{d}{d\lambda}(\hat{V}'_\lambda(u))$ and $u$ have the opposite sign.

Now, fix $\bar{\eta}$, $\lambda_1<\lambda_2$. Since $\frac{du}{d\lambda}$ and $u$ has the same sign, we have $u^-_{\bar{\eta},\lambda_2}<u^-_{\bar{\eta},\lambda_1}<0<u^+_{\bar{\eta},\lambda_1}<u^+_{\bar{\eta},\lambda_2}$. Consider the function $\hat{V}_{\lambda_1}(u)$ and $\hat{V}_{\lambda_2}(u+u^+_{\bar{\eta},\lambda_2}-u^+_{\bar{\eta},\lambda_1})$, Both of them have the same value $\bar{\eta}$ at $u=u^+_{\bar{\eta},\lambda_1}$. Now, for all fixed $\tilde{\eta}\in(0,\bar{\eta})$, $\frac{d}{d\lambda}(\hat{V}'_\lambda(u))$ and $u$ have the opposite sign, we have
\begin{equation}
\frac{\partial u^+_{\tilde{\eta},\lambda_1}}{\partial\tilde{\eta}}=\frac{1}{\hat{V}'_{\lambda_1}(u^+_{\tilde{\eta},\lambda_1})}<\frac{1}{\hat{V}'_{\lambda_2}(u^+_{\tilde{\eta},\lambda_2})}=\frac{\partial u^+_{\tilde{\eta},\lambda_2}}{\partial\tilde{\eta}}.
\end{equation}
Therefore, for any $\tilde{\eta}\in(0,\bar{\eta})$, $u^+_{\tilde{\eta},\lambda_1}>u^+_{\tilde{\eta},\lambda_2}-(u^+_{\bar{\eta},\lambda_2}-u^+_{\bar{\eta},\lambda_1})$, i.e. the graph of $(u^+_{\tilde{\eta},\lambda_1},\tilde{\eta})$ lies on the right of the graph of $(u^+_{\tilde{\eta},\lambda_2}-(u^+_{\bar{\eta},\lambda_2}-u^+_{\bar{\eta},\lambda_1}),\tilde{\eta})$. Since $\hat{V}'_\lambda(u^+_{\tilde{\eta},\lambda})>0$, it implies that $\hat{V}_{\lambda_1}(u)<\hat{V}_{\lambda_2}(u+u^+_{\bar{\eta},\lambda_2}-u^+_{\bar{\eta},\lambda_1})$ for $u\in(0, u^+_{\bar{\eta},\lambda_1})$. We can employ the same argument for the negative part and obtain $\hat{V}_{\lambda_1}(u)<\hat{V}_{\lambda_2}(u+u^-_{\bar{\eta},\lambda_2}-u^-_{\bar{\eta},\lambda_1})$ for $u\in(u^-_{\bar{\eta},\lambda_1},0)$.
 
Therefore,
\begin{equation}
\begin{split}
\triangle\theta(\bar{\eta},\lambda_1)=&\int_{u^-_{\bar{\eta},\lambda_1}}^{u^+_{\bar{\eta},\lambda_1}}
\frac{2du}{\sqrt{\bar{\eta}-\hat{V}_{\lambda_1}(u)}}
=\int_{u^-_{\bar{\eta},\lambda_1}}^0
\frac{2du}{\sqrt{\bar{\eta}-\hat{V}_{\lambda_1}(u)}}+\int_0^{u^+_{\bar{\eta},\lambda_1}}
\frac{2du}{\sqrt{\bar{\eta}-\hat{V}_{\lambda_1}(u)}}\\
<&\int_{u^-_{\bar{\eta},\lambda_1}}^0
\frac{2du}{\sqrt{\bar{\eta}-\hat{V}_{\lambda_2}(u+u^-_{\bar{\eta},\lambda_2}-u^-_{\bar{\eta},\lambda_1})}}\\
&+\int_0^{u^+_{\bar{\eta},\lambda_1}}
\frac{2du}{\sqrt{\bar{\eta}-\hat{V}_{\lambda_2}(u+u^+_{\bar{\eta},\lambda_2}-u^+_{\bar{\eta},\lambda_1})}}\\
=&\int_{u^-_{\bar{\eta},\lambda_2}}^{u^-_{\bar{\eta},\lambda_2}-u^-_{\bar{\eta},\lambda_1}}
\frac{2du}{\sqrt{\bar{\eta}-\hat{V}_{\lambda_2}(u)}}+\int_{u^+_{\bar{\eta},\lambda_2}-u^+_{\bar{\eta},\lambda_1}}^{u^+_{\bar{\eta},\lambda_2}}
\frac{2du}{\sqrt{\bar{\eta}-\hat{V}_{\lambda_2}(u)}}\\
<&\int_{u^-_{\bar{\eta},\lambda_2}}^{u^+_{\bar{\eta},\lambda_2}}
\frac{2du}{\sqrt{\bar{\eta}-\hat{V}_{\lambda_2}(u)}}=\triangle\theta(\bar{\eta},\lambda_2).
\end{split}
\end{equation}
\end{proof}


Now, compare the case $\lambda>0$ with the self shrinker case.

\begin{proof}[Proof of theorem \ref{poslambda}]
Use the result in \cite{AL} that when $\lambda=0$, $\triangle\theta(\bar{\eta},0)$ is a decreasing function of $\bar{\eta}$, $\lim_{\bar{\eta}\to0}\triangle\theta(\eta,0) =\sqrt{2}\pi$, $\lim_{\bar{\eta}\to\infty}\triangle\theta(\eta,0) =\pi$. Hence for all $\bar{\eta}$, $\triangle\theta(\eta,0)>\pi$. Plugging $\lambda_1=0$, $\lambda_2=\lambda$ in the previous theorem, we get $
\triangle\theta(\bar{\eta},\lambda)>\triangle\theta(\bar{\eta},0)>\pi$ for $\lambda>0$. 
\end{proof}

\section{Simulation of the curves}
By using Matlab program to solve the ODE system, we can obtain numerical solutions which approximately solve the ODE. The following are the curves of the numerical solutions. The curves behave as what expected from the theorem \ref{poslambda}, theorem \ref{neglambdanear0} and theorem \ref{neglambdanearinfty}. We put the pictures of the curves here to give the reader better idea of what the actual solution will behave. From the simulation, we can observe some behavior of $\triangle\theta$ with respect to $\eta$. Some conjectures about the behavior are posed here.

When $\lambda>0$, the range for $\triangle\theta$ contains $(\pi, \pi\sqrt{2}\sqrt{\frac{\lambda}{\sqrt{\lambda^2+4}}+1}]$ and $\triangle\theta>\pi$. There will not be embedded solutions. The following are some of the closed solutions for the case $\lambda=0.19$ and $\lambda=0.726$. The energy $\eta$ increases from left to right. Note that for certain $\eta$, the solution passes through the origin. If we keep increasing $\eta$, unlike the case where $\lambda=0$, the origin will not be on the same side of the solution anymore. Also, we can conjecture that $\triangle\theta$ is decreasing when $\eta$ is increasing, as in the case for self-shrinkers.

\begin{con} When $\lambda>0$, $\triangle\theta$ is monotonically decreasing with respect to $\eta$ as in the case of Abresch and Langer\cite{AL}.
\end{con}

\begin{figure}[ht]
\centering
\includegraphics[width=90pt]{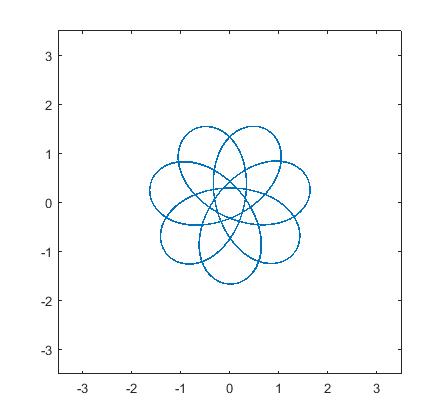}
\includegraphics[width=90pt]{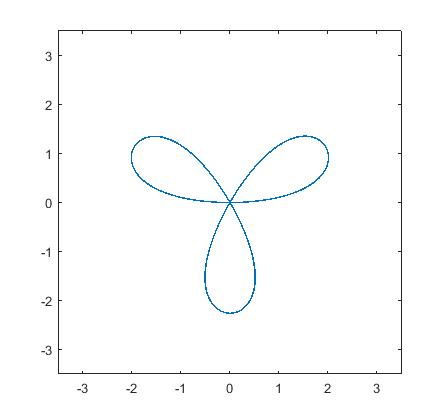}
\includegraphics[width=90pt]{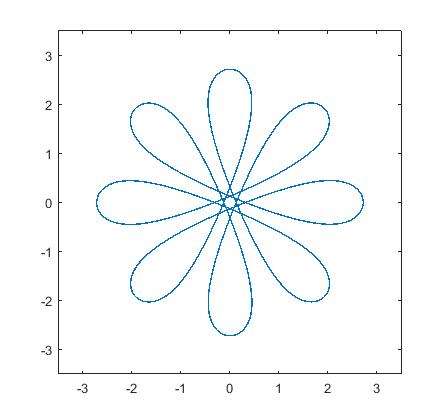}
\includegraphics[width=90pt]{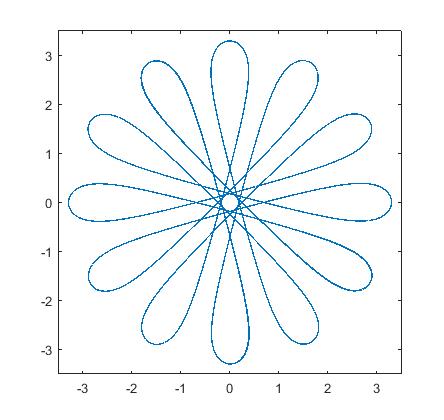}
\caption{Solutions for $\lambda=0.19$, $\triangle\theta=\frac{10\pi}{7}$, $\frac{4\pi}{3}$, $\frac{5\pi}{4}$, $\frac{7\pi}{6}$, respectively.}
\end{figure}

\begin{figure}[ht]
\centering
\includegraphics[width=90pt]{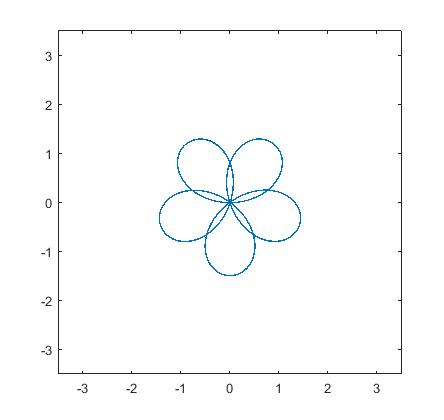}
\includegraphics[width=90pt]{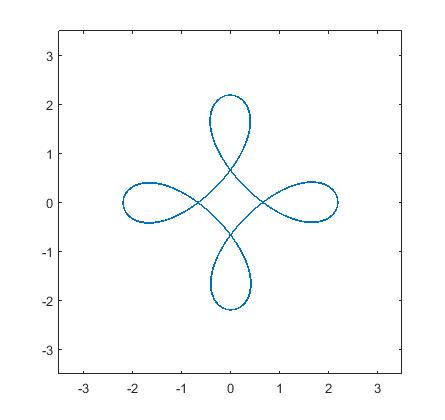}
\includegraphics[width=90pt]{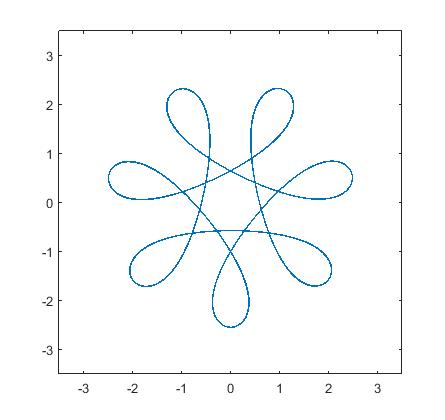}
\includegraphics[width=90pt]{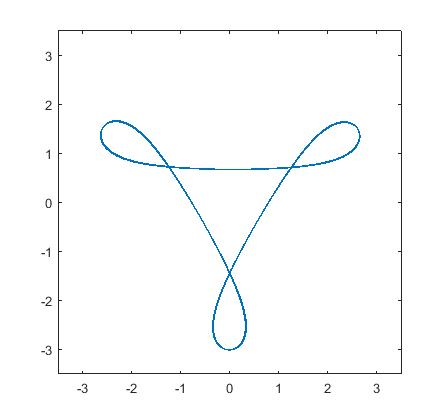}
\caption{Solutions for $\lambda=0.726$, $\triangle\theta=\frac{8\pi}{5}$, $\frac{3\pi}{2}$, $\frac{10\pi}{7}$, $\frac{4\pi}{3}$, respectively.}
\end{figure}

\bigskip
\bigskip
\bigskip
\bigskip
\bigskip
\bigskip
\bigskip
\bigskip
\bigskip
\bigskip
\bigskip
The $\lambda<0$ cases are more interesting. For each $\frac{-2}{\sqrt{3}}<\lambda<0$, there exist $\eta$ such that $\triangle\theta=\pi$ and the solution is embedded and have 2-symmetry. The following are some of the examples. From left to right, $\lambda=$-0.2, -0.3, -0.4, -0.5, -0.6, -0.7, -0.8, -0.9, respectively.

\begin{con}
There is a unique $\delta>0$ such that for $\frac{-7}{2\sqrt{2}}+\delta<\lambda\leq\frac{-2}{\sqrt{3}}$, there are no embedded solutions and there is a 3-symmetry solution when $\lambda\leq\frac{-7}{2\sqrt{2}}+\delta$.
\end{con}

\begin{figure}[h]
\centering
\includegraphics[width=30pt]{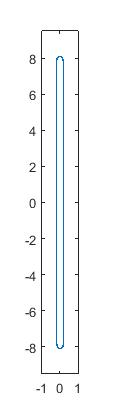}
\includegraphics[width=30pt]{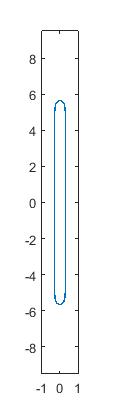}
\includegraphics[width=30pt]{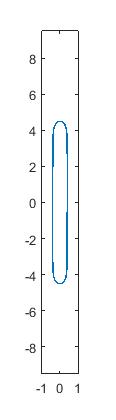}
\includegraphics[width=30pt]{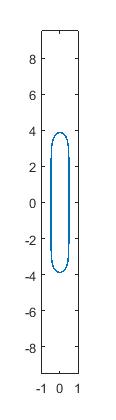}
\includegraphics[width=30pt]{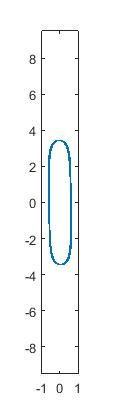}
\includegraphics[width=30pt]{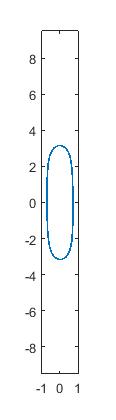}
\includegraphics[width=30pt]{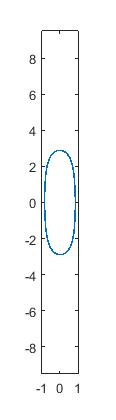}
\includegraphics[width=30pt]{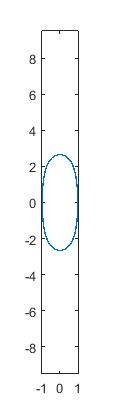}
\caption{Embedded solutions for different $\lambda$'s}
\end{figure}

For $\lambda\leq\frac{-2}{\sqrt{3}}$, we have no embedded solutions with 2-symmetry. However, as $\lambda<\frac{-7}{2\sqrt{2}}$, there are embedded solutions with $m$-symmetry, $m>2$. The following are the cases where $\lambda=$-3, -5. The energy $\eta$ increases from left to right. Unlike the case $\lambda>0$, when $\lambda\leq\frac{-2}{\sqrt{3}}$, even though $\triangle\theta$ should be decreasing near $\min V(B)$, it appears that after some $\eta$, $\triangle\theta$ is increasing while $\eta$ is increasing.

\begin{con}
When $\lambda<0$, there is a function $\eta_{crit}(\lambda)$ such that $\triangle\theta$ is decreasing when $\eta<\eta_{crit}$ and $\triangle\theta$ is increasing when $\eta>\eta_{crit}(\lambda)$. Moreover, $\eta_{crit}(\lambda)$ goes to infinity when $\lambda$ goes to zero, $\eta_{crit}(\lambda)-\min V_\lambda$ goes to zero when $\lambda$ goes to negative infinity.
\end{con}

\begin{figure}[ht]
\centering
\includegraphics[width=90pt]{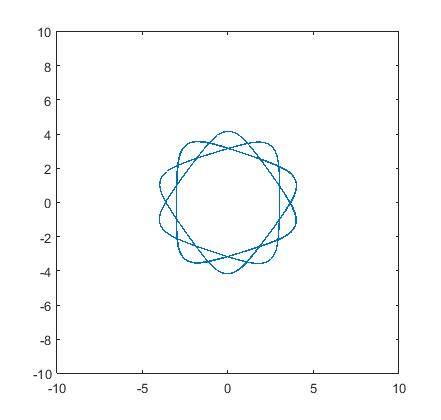}
\includegraphics[width=90pt]{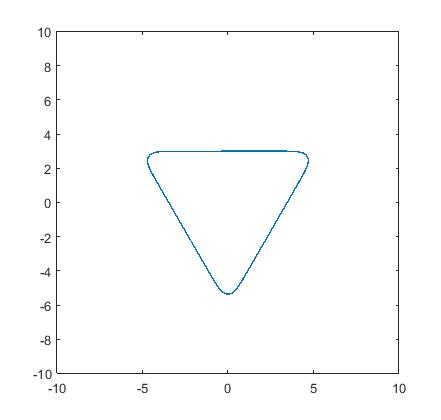}
\includegraphics[width=90pt]{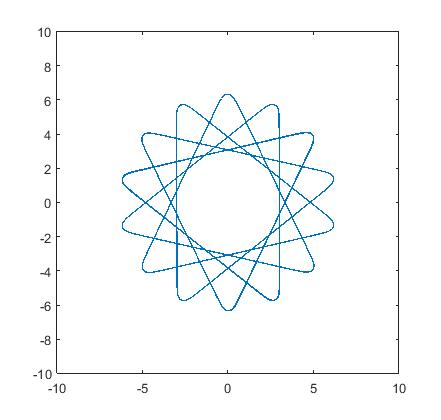}
\includegraphics[width=90pt]{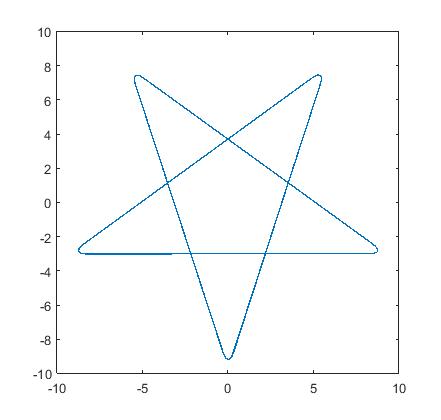}
\caption{Solutions for $\lambda=-3$, $\triangle\theta=\frac{3\pi}{5}$, $\frac{2\pi}{3}$, $\frac{5\pi}{7}$, $\frac{4\pi}{5}$, respectively.}
\end{figure}

\begin{figure}[ht]
\centering
\includegraphics[width=90pt]{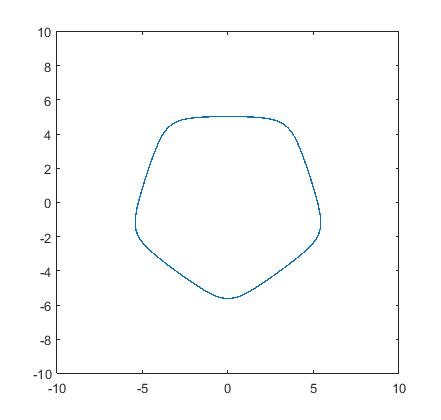}
\includegraphics[width=90pt]{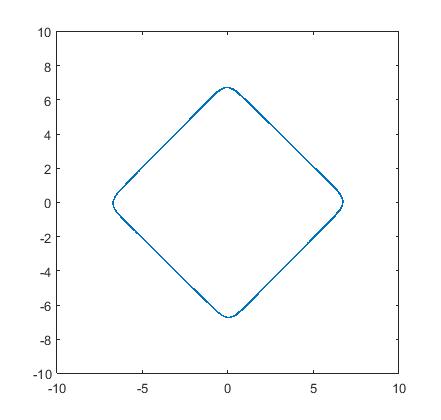}
\includegraphics[width=90pt]{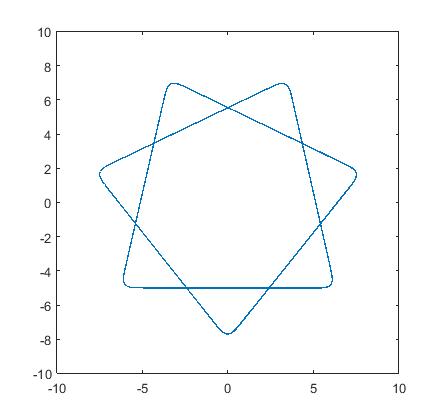}
\includegraphics[width=90pt]{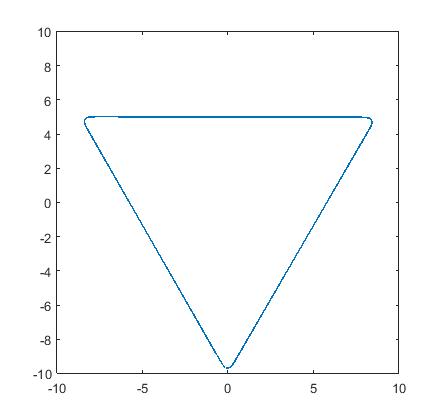}
\caption{Solutions for $\lambda=-5$, $\triangle\theta=\frac{2\pi}{5}$, $\frac{\pi}{2}$, $\frac{4\pi}{7}$, $\frac{2\pi}{3}$, respectively.}
\end{figure}

\end{document}